\documentclass{nyj}
\usepackage{dsfont}
\input xy
\xyoption{all}

\newtheorem{theorem}{Theorem}[section]
\newtheorem{lemma}[theorem]{Lemma}
\newtheorem{proposition}[theorem]{Proposition}

\theoremstyle{definition}
\newtheorem{definition}[theorem]{Definition}

\newtheorem{remark}[theorem]{Remark}
\newtheorem{corollary}[theorem]{Corollary}

\numberwithin{equation}{section}

\title{$C^*$-algebra of the $\mathds{Z}^n$-tree}

\author{Menassie Ephrem}  
\address{Department of Mathematics and Statistics, Coastal Carolina
University, Conway, SC 29528-6054}
\email{menassie@coastal.edu}   



\begin{document}


\date{January, 2008.}


\begin{abstract}

Let $\Lambda = \mathbb{Z}^n$ with lexicographic ordering.  $\Lambda$ is
a totally ordered group.  Let $X = \Lambda^+
* \Lambda^+$. Then $X$ is a $\Lambda$-tree. Analogous to the
construction of graph $C^*$-algebras, we form a groupoid whose unit
space is the space of ends of the tree. The $C^*$-algebra of the
$\Lambda$-tree is defined as the $C^*$-algebra of this groupoid. We prove some properties of this $C^*$-algebra.

\end{abstract}

\maketitle

\tableofcontents

\section{Introduction} \label{sect.z21}
Since the introduction of $C^*$-algebras of groupoids, in the late
1970's, several classes of $C^*$-algebras have been given groupoid
models.  One such class is the class of graph $C^*$-algebras.

In their paper \cite{KPRR}, Kumjian, Pask, Raeburn and Renault
associated to each locally finite directed graph $E$ a locally
compact groupoid $\mathcal{G}$, and showed that its groupoid
$C^*$-algebra $C^*(\mathcal{G})$ is the universal $C^*$-algebra
generated by families of partial isometries satisfying the
Cuntz-Krieger relations determined by $E$.  In \cite{Sp1}, Spielberg constructed a locally compact groupoid $\mathcal{G}$ associated to
a general graph $E$ and generalized the result to a general
directed graph.

We refer to \cite{R} for the detail theory of topological
groupoids and their $C^*$-algebras.

A directed graph $E = (E^0, E^1,~~ o,~~ t)$ consists of a
countable set $E^0$ of vertices and $E^1$ of edges, and maps $o,t:
E^1 \rightarrow E^0$ identifying the origin (source) and the
terminus (range) of each edge.
For the purposes of this discussion it is sufficient to consider
row--finite graphs with no sinks.

For the moment, let $T$ be a bundle of of row--finite directed trees with
no sinks, that is a disjoint union of trees that have no sinks or infinite emitters, i.e., no singular vertices.  We denote the set of finite paths of $T$ by
$T^*$ and the set of infinite paths by $\partial T$.

For each $p \in T^*$, define $$V(p):=\{px:x \in \partial T,~~ t(p) =
o(x)\}.$$
For $p, q \in T^*$, we see that:
\begin{displaymath}
V(p) \cap V(q) = \left\{ \begin{array}{ll}
        V(p) & \textrm{if $p = qr$ for some $r \in T^*$}\\
        V(q) & \textrm{if $q = pr$ for some $r \in T^*$}\\
        \emptyset & \textrm{otherwise}.
        \end{array} \right.
\end{displaymath}
It is fairly easy to see that:
\begin{lemma}
The cylinder sets $\{V(p):p \in T^*\}$ form a base of compact open
sets for a locally compact, totally disconnected, Hausdorff
topology of $\partial T$.
\end{lemma}

We want to define a groupoid that has $\partial T$ as a unit space. For $x =
x_1x_2\ldots$, and $y = y_1y_2 \ldots \in \partial T$, we say
\textit{$x$ is shift equivalent to $y$ with lag $k \in
\mathds{Z}$} and write $x \sim_k y$, if there exists $n \in
\mathds{N}$ such that $x_i = y_{k+i} ~~ \textrm{for each } i \geq
n$.
It is not difficult to see that shift equivalence is an
equivalence relation.

\begin{definition} \label{def.501}
Let $\mathcal{G} := \{(x,k,y) \in  \partial T \times \mathds{Z}
\times \partial T : x \thicksim_k y \}$.
For pairs in $\mathcal{G}^2 := \{((x,k,y),(y,m,z)) :
(x,k,y),(y,m,z) \in \mathcal{G}\}$, we define
\begin{equation}
\label{eq01:eps} (x,k,y) \cdot (y,m,z) = (x, k+m, z).
\end{equation}
For arbitrary $(x,k,y) \in \mathcal{G}$, we define
\begin{equation} \label{eq02:eps}(x,k,y)^{-1} = (y, -k, x). \end{equation}
\end{definition}
With the operations \eqref{eq01:eps} and \eqref{eq02:eps}, and
source and range maps $s,r:\mathcal{G} \longrightarrow \partial T$
given by $s(x,k,y) = y,~~ r(x,k,y) = x,~~ \mathcal{\mathcal{G}}$
is a groupoid with unit space $\partial T$.

For $p,q \in T^*$, with $t(p) = t(q)$, define $U(p,q) := \{px,
l(p)-l(q),qx):x \in \partial T,~~ t(p) = o(x)\}$, where $l(p)$
denotes the length of the path $p$.
The sets $\{U(p,q):p,q\in T^*, t(p) = t(q)\}$ make $\mathcal{G}$ a
locally compact $r$-discrete groupoid with (topological) unit
space equal to $\partial T$.

Now let $E$ be a directed graph.  We form a graph whose
vertices are the paths of $E$ and edges are (ordered) pairs of
paths as follows:

\begin{definition}
Let $\widetilde{E}$ denote the following graph:

$\widetilde{E}^0 = E^*$

$\widetilde{E}^1 = \{(p,q) \in E^* \times E^* : q = pe~~ for~~
some~~ e \in E^1\}$

$o(p,q) = p,~~ t(p,q) = q$.
\end{definition}

The following lemma, due to Spielberg \cite{Sp1}, is straightforward.

\begin{lemma} \label{lemm.sp24} \cite[Lemma 2.4]{Sp1}
$\widetilde{E}$ is a bundle of trees.
\end{lemma}

Notice that if $E$ is a row--finite graph with no sinks, then
$\widetilde{E}$ is a bundle of row--finite trees with no sinks.

If $\mathcal{G}(E)$ is the groupoid obtained as in Definition
\ref{def.501}, where $\widetilde{E}$ plays the role of $T$ then, in \cite{Sp1} Spielberg showed, in its full generality, that
the graph $C^*$-algebra of $E$ is equal to the $C^*$-algebra
of the groupoid $\mathcal{G}(E)$.  We refer to \cite{Sp1}, for
readers interested in the general construction and the proof.

We now examine the $C^*$-algebra $\mathcal{O}_2$, which is
the $C^*$-algebra of the graph $$E\hspace{.3in} =\hspace{.3in}
\xymatrix{0 \ar@(dl,dr)[]_a \ar@(ru,ul)[]_b}$$
Denoting the vertex of $E$ by 0 and the edges of $E$ by $a$ and
$b$, as shown in the graph, the vertices of $\widetilde{E}$ are
$0, a, b, aa, ab, ba, bb,$ etc.  And the graph $\widetilde{E}$ is
the binary tree.

Take a typical path $p$ of $E$, say $p = aaabbbbbabbaaaa$.  Writing $aaa$
as 3 and $bbbbb$ as $5'$, etc. we can write $p$ as $35'12'4$ which is an
element of $\mathds{Z}^+ * \mathds{Z}^+$ (the free product of two
copies of $\mathds{Z}^+$). In other words, the set of vertices of
$\widetilde{E}$ is $G_1^+ * G_2^+$, where $G_1^+ = \mathds{Z}^+ =
G_2^+$, and the vertex 0 is the empty word.  The elements of $\partial \widetilde{E}$ are the infinite sequence
of $n$'s and $m$'s, where $n \in G_1^+$ and $m \in G_2^+$.

Motivated by this construction, we wish to explore the
$C^*$-algebra of the case when $\Lambda$ is an ordered abelian
group, and $X$ is the free product of two copies of $\Lambda^+$.  In this paper we study the special case when $\Lambda =
\mathds{Z}^n$ endowed with the lexicographic ordering, where $n \in \{2,~~ 3,~~ \ldots\}$.

The paper is organized as follows.  In section 2 we develop the topology of the $\mathds{Z}^n$-tree.  In section 3 we build the $C^*$-algebra of the $\mathds{Z}^n$-tree by first building the groupoid $\mathcal{G}$ in a fashion similar to that of the graph groupoid.  In section 4, by explicitly exploring the partial isometries generating the $C^*$-algebra, we give a detailed description of the $C^*$-algebra.  In section 5, we look at the crossed product of the $C^*$-algebra by the gauge action and study the fixed-point algebra.  Finally in section 6 we provide classification of the $C^*$-algebra. We prove that the $C^*$-algebra is simple, purely infinite, nuclear and
classifiable.

I am deeply indebted to Jack Spielberg without whom none of this would have been possible.  I also wish to thank Mark Tomforde for many helpful discussions and for providing material when I could not find them otherwise.

\section{The $\mathds{Z}^n$-tree and its boundary} \label{sect.z22}

Let $n \in \{2,~~ 3,~~, \ldots\}$ and let $\Lambda = \mathds{Z}^n$ together with lexicographic ordering,
that is, $(k_1, k_2, \ldots, k_n) < (m_1, m_2, \ldots, m_n)$ if either $k_1 < m_1$, or
$k_1 = m_1$, $\ldots$, $k_{d-1} = m_{d-1}$, and $k_d < m_d$.  We denote by $\partial\Lambda^+$
the set $\{(k_1, k_2, \ldots, k_{n-1}, \infty):k_i \in \mathds{N} \cup \{\infty\}, ~ k_i = \infty \Rightarrow k_{i+1} = \infty\}$.

Let $G_i = \Lambda^+ = \{a \in \Lambda : a > 0\}$ for $i =
1,~~2$, and let $\partial G_i = \partial \Lambda^+$ for $i = 1,~~2$.
That is, we take two copies of $\Lambda^+$ and label them as $G_1$
and $G_2$, and two copies of $\partial\Lambda^+$ and label them as
$\partial G_1$ and $\partial G_2$.
Now consider the set $X = G_1 * G_2$.  We denote the empty word by
0.
Thus, $\displaystyle X = \bigcup_{d=1}^\infty \{a_1a_2 \ldots a_d : a_i \in G_k \Rightarrow a_{i+1}
\in G_{k\pm 1}\text{ for }  1
\leq i < d\}~~ \bigcup ~~ \{0\}$.  We note that $X$ is a $\Lambda$-tree, as studied in \cite{Ch}.

Let $\partial X = \{a_1a_2 \ldots a_d : a_i \in G_k
\Rightarrow a_{i+1} \in G_{k\pm 1},\text{ for } 1\leq i < d-1\text{ and }  a_{d-1} \in
G_k \Rightarrow a_d \in \partial G_{k\pm 1} \}~~ \bigcup ~~ \{a_1a_2 \ldots : a_i \in G_k \Rightarrow a_{i+1}
\in G_{k\pm 1}\text{ for each } i\}$.  In words, $\partial X$ contains either a finite sequence of
elements of $\Lambda$ from sets with alternating indices, where the last
element is from $\partial \Lambda^+$, or an infinite sequence of
elements of $\Lambda$ from sets with alternating indices.\\
For $a \in \Lambda^+$ and $b \in \partial
\Lambda^+$, define $a+b \in \partial \Lambda^+$ by componentwise addition.

For $p = a_1a_2 \ldots a_k \in X$ and $q = b_1b_2\ldots b_m \in X
\cup \partial X$, i.e., $m \in \mathds{N} \cup \{\infty\}$, define $pq$ as
follows:

\noindent i) If $a_k, b_1 \in G_i \cup \partial G_i$ (i.e., they
belong to sets with the same index), then $pq := a_1a_2 \ldots
a_{k-1} (a_k+b_1)b_2 \ldots b_m$. Observe that since $a_k \in
\Lambda$, the sum $a_k+b_1$ is defined and is in the same set as
$b_1$.

\noindent ii) If $a_k$ and  $b_1$ belong to sets with different
indices, then $pq := a_1a_2 \ldots a_kb_1b_2 \ldots b_m$.

\noindent In other words, we concatenate $p$ and $q$ in the most
natural way (using the group law in $\Lambda * \Lambda$).

For $p \in X$ and $q \in X \cup \partial X$, we write $p \preceq
q$ to mean $q$ extends $p$, i.e., there exists $r \in X \cup
\partial X$ such that $q = pr$.

For $p \in \partial X$ and $q \in X \cup \partial X$, we write $p
\preceq q$ to mean $q$ extends $p$, i.e., $\textrm{for each } r
\in X$, $r \preceq p$ implies that $r \preceq q$.

We now define two length functions. Define $l:X \cup \partial X \longrightarrow (\mathds{N} \cup
\{\infty\})^n$ by $l(a_1a_2\ldots
a_k) := \sum_{i=1}^ka_i$.

And define $l_i:X \cup \partial X \longrightarrow \mathds{N} \cup
\{\infty\}$ to be the $i^{th}$ component of $l$, i.e., $l_i(p)$ is
the $i^{th}$ component of $l(p)$.  It is easy to see that both $l$ and $l_i$ are additive.

Next, we define basic open sets of $\partial X$. For $p, q \in X$, we define $V(p) := \{px:x \in \partial X\}$ and
$V(p;q) := V(p) \setminus V(q)$.

Notice that

\begin{equation}\label{eq11:eps}
V(p) \cap V(q) = \left\{ \begin{array}{ll}
        \emptyset & \textrm{if $p \npreceq q$ and $q \npreceq p$}\\
        V(p) & \textrm{if $q \preceq p$}\\
        V(q) & \textrm{if $p \preceq q$}.
        \end{array} \right.
\end{equation}

Hence

\begin{displaymath}
V(p) \setminus V(q) = \left\{ \begin{array}{ll}
        V(p) & \textrm{if $p \npreceq q$ and $q \npreceq p$}\\
        \emptyset & \textrm{if $q \preceq p$}.
        \end{array} \right.
\end{displaymath}

Therefore, we will assume that $p \preceq q$ whenever we write
$V(p;q)$.\\

Let $\mathcal{E} := \{V(p):p\in X\}~~ \bigcup~~ \{V(p;q):p,q \in
X\}$.

\begin{lemma}\label{Eseparats.points}
$\mathcal{E}$ separates points of $\partial X$, that is, if $x,y \in \partial X$ and $x\neq y$ then there exist two sets $A,B \in \mathcal{E}$ such that $x \in A$, $y\in B$, and $A\cap B = \emptyset$.
\end{lemma}

\begin{proof}

Suppose $x,y \in \partial X$ and $x \neq y$.  Let $x = a_1a_2\ldots a_s$, $y = b_1b_2\ldots b_m$.  Assume, without loss of generality, that $s \leq m$.  We consider two cases:

\noindent Case I. there exists $k < s $ such that $a_k \neq b_k$ (or they
belong to different $G_i$'s).  Then $x \in V(a_1a_2\ldots a_k),
~~y \in V(b_1b_2\ldots b_k)$ and $V(a_1a_2\ldots a_k) \cap
V(b_1b_2\ldots b_k) = \emptyset$.

\noindent Case II. $a_i = b_i ~~ \textrm{for each } i < s $.  Notice that
if $s = \infty$, that is, if both $x$ and $y$ are infinite sequences
then there should be a $k \in \mathds{N}$ such that $a_k \neq b_k$
which was considered in case I.  Hence $s < \infty$.  Again, we
distinguish two subcases:

\begin{enumerate}

\item [a)] $s = m$. Therefore $x = a_1a_2\ldots a_s$ and $y =
a_1a_2\ldots b_s$, and $a_s,b_s \in \partial G_i$, with $a_s \neq
b_s$.  Assuming, without loss of generality, that $a_s < b_s$, let
$a_s = (k_1, k_2, \ldots, k_{n-1}, \infty)$, and $b_s = (r_1, r_2, \ldots, r_{n-1}, \infty)$ where $(k_1, k_2, \ldots, k_{n-1}) <
(r_1, r_2, \ldots, r_{n-1})$.  Therefore there must be an index $i$ such that $k_i < r_i$;  let $j$ be the largest such.  Hence $a_s + e_j \leq b_s$, where $e_j$ is the $n$-tuple with 1 at the $j^{th}$ spot and 0 elsewhere.
Letting $c = a_s + e_j$, we see that $x \in A = V(a_1a_2\ldots
a_{s-1};c),~~ y \in B = V(c)$, and $A \cap B  = \emptyset$.

\item [b)] $s < m$.  Then $y = a_1a_2 \ldots
a_{s-1}b_sb_{s+1}\ldots b_m$ ($m \geq s +1$).

Since $b_{s+1} \in (G_i \cup \partial G_i) \setminus \{0\}$ for $i
= 1,2$, choose $c = e_n \in G_i$ (same index as $b_{s+1}$ is
in). Then $x \in A = V(a_1a_2\ldots a_{s-1};a_1a_2\ldots
a_{s-1}b_sc)$, $y \in B = V(a_1a_2\ldots a_{s-1}b_sc)$, and $A
\cap B = \emptyset$.

\end{enumerate}

\noindent This completes the proof.\\
\end{proof}

\begin{lemma}
$\mathcal{E}$ forms a base of compact open sets for a locally
compact Hausdorff topology on $\partial X$.
\end{lemma}
\begin{proof}
First we prove that $\mathcal{E}$ forms a base. Let $A = V(p_1;p_2)$ and $B=V(q_1;q_2)$.  Notice that if $p_1 \npreceq q_1$ and $q_1 \npreceq p_1$ then $A \cap B = \emptyset$. Suppose,
without loss of generality, that $p_1 \preceq q_1$ and let $x \in A
\cap B$.  Then by construction, $p_1 \preceq q_1 \preceq x$ and $p_2 \npreceq x$ and $q_2 \npreceq x$.  Since $p_2 \npreceq x$ and $q_2 \npreceq x$, we can choose $r \in
X$ such that $q_1 \preceq r$, $p_2 \npreceq r$, $q_2
\npreceq r$, and $x = ra$ for some $a \in \partial X$. If $x \npreceq p_2$ and $x \npreceq q_2$ then $r$ can be chosen so
that $r \npreceq p_2$ and $r \npreceq q_2$, hence $x \in V(r)
\subseteq A \cap B$.

Suppose now that $x \preceq p_2$.  Then $x = ra$, for some $r \in X$ and $a \in \partial X$.  By extending $r$ if necessary, we may assume that $a \in \partial \Lambda^+$.  Then we may write $p_2 = rby$ for some $b \in \Lambda^+$, and $y \in \partial X$ with $a < b$.  Let $b' = b - (0,\ldots, 0, 1)$, and
$s_1 = rb'$.  Notice that $x \preceq s_1
\preceq p_2$ and $s_1 \neq p_2$.  If $x \npreceq q_2$ then we can choose $r$ so that $r \npreceq q_2$.  Therefore $x \in V(r;s_1) \subseteq A \cap B$. If $x \preceq q_2$, construct $s_2$ the way as $s_1$ was
constructed, where $q_2$ takes the place of $p_2$.  Then either $s_1
\preceq s_2$ or $s_2 \preceq s_1$.  Set  \begin{displaymath}
s = \left\{ \begin{array}{ll}
        s_1 & \textrm{if $s_1 \preceq s_2$}\\
        s_2 & \textrm{if $s_2 \preceq s_1$}.
        \end{array} \right.
\end{displaymath} Then $x \in V(r;s) \subseteq A \cap B$. The cases when $A$ or $B$ is of the form $V(p)$ are similar, in fact easier.

That the topology is Hausdorff follows from the fact that $\mathcal{E}$ separates points.

Next we prove local compactness.  Given $p,q \in X$ we need to prove that $V(p;q)$ is compact. Since
$V(p;q) = V(p) \setminus V(q)$ is a (relatively) closed subset of
$V(p)$, it suffices to show that $V(p)$ is compact.
Let $A_0 = V(p)$ be covered by an open cover $\mathcal{U}$
and suppose that $A_0$ does not admit a finite subcover. Choose
$p_1 \in X$ such that $l_i(p_1) \geq 1$ and $V(pp_1)$ does not
admit a finite subcover, for some $i \in \{1,\ldots, n-1\}$.  We consider two cases:

\noindent Case I.  Suppose no such $p_1$ exists. Let $a = e_n \in G_1,~~ b = e_n \in G_2$.  Then $V(p) = V(pa)
\cup V(pb)$.  Hence either $V(pa)$ or $V(pb)$ is not finitely
covered, say $V(pa)$, then let $x_1 = a$. After choosing $x_s$, since $V(px_1\ldots x_s) = V(px_1\ldots
x_sa) \cup V(px_1\ldots x_sb)$, either $V(px_1\ldots x_sa)$ or
$V(px_1\ldots x_sb)$ is not finitely covered. And we let $x_{s+1}
= a$ or $b$ accordingly.
Now let $A_j = V(px_1\ldots x_j)$ for $j \geq 1$ and let $x =
px_1x_2\ldots \in \partial X$.  Notice that $A_0 \supseteq A_1
\supseteq A_2 \ldots$, and $x \in \bigcap_{j=0}^\infty A_j$.
Choose $A' \in \mathcal{U}$, $q,r \in X$, such that $x \in V(q;r) \subseteq A'$. Clearly
$q \preceq x$ and $r \npreceq x$. Once again, we distinguish two subcases:
\begin{enumerate}

\item [a)] $x \npreceq r$.  Then, for a large enough $k$ we get $q
\preceq px_1x_2\ldots x_k$ and $px_1x_2\ldots x_k \npreceq r$.
Therefore $A_k = V(px_1x_2\ldots x_k) \subseteq A'$, which
contradicts to that $A_k$ is not finitely covered.

\item [b)] $x \preceq r$.  Notice $l_1(x) = l_1(p)$ and since $x =
px_1x_2 \ldots \preceq r$, we have $l_1(x) = l_1(p) < l_1(r)$.  Therefore
$V(r)$ is finitely covered, say by $B_1,B_2,\ldots, B_s \in
\mathcal{U}$. For large enough $k$, $q \preceq px_1x_2 \ldots
x_k$. Therefore $A_k = V(px_1x_2 \ldots x_k) \subseteq V(q) =
V(q;r) \cup V(r) \subseteq A' \cup \bigcup_{j=1}^nB_j$, which is a
finite union.  This is a contradiction.

\end{enumerate}

\noindent Case II. Let $p_1 \in X$ such that $l_i(p_1) \geq 1$ and $V(pp_1)$
is not finitely covered, for some $i \in \{1,\ldots, n-1\}$.  After choosing $p_1, \ldots, p_s$ let $p_{s+1}$ be such that
$l_i(p_{s+1}) \geq 1$ and $V(pp_1\ldots p_{s+1})$ is not finitely
covered, for some $i \in \{1,\ldots, n-1\}$.  If no such $p_{s+1}$ exists then we are back in to case
I with $V(pp_1p_2\ldots p_s)$ playing the role of $V(p)$.
Now let $x = pp_1p_2\ldots \in \partial X$ and let $A_j =
V(pp_1\ldots p_j)$.  We get $A_0 \supseteq A_1 \supseteq \ldots$, and
$x = pp_1p_2 \ldots \in \bigcap_{j=0}^\infty A_j$.  Choose $A' \in \mathcal{U}$ such that $x \in V(q;r) \subseteq A'$.
Notice that $q \preceq x$ and $n-1$ is finite, hence there exists $i_0 \in \{1,\ldots,n-1\}$ such that $l_{i_0}(x) = \infty$.  Since $l_{i_0}(r) < \infty$, we have $x \npreceq r$.  Therefore, for large
enough $k$, $q \preceq pp_1\ldots p_k \npreceq r$, implying
$A_k \subseteq A'$, a contradiction.

\noindent Therefore $V(p)$ is compact.
\end{proof}

\section{The groupoid and $C^*$-algebra of the
$\mathds{Z}^n$-tree} \label{sect.z23}

We are now ready to form the groupoid which will eventually be
used to construct the $C^*$-algebra of the $\Lambda$-tree.

For $x,y \in \partial X$ and $k \in \Lambda$, we write $x
\thicksim_k y$ if there exist $p,q \in X$ and $z \in \partial X$
such that $k = l(p) - l(q)$ and $x = pz, y = qz$.

Notice that:

\begin{enumerate}

\item If $x \thicksim_k y$ then $y \thicksim_{-k} x$.

\item $x \thicksim_0 x$.

\item If $x \thicksim_k y$ and $y \thicksim_m z$ then $x = \mu t,
~~ y = \nu t, ~~ y = \eta s,~~ z = \beta s$ for some $\mu, \nu,
\eta, \beta \in X~~ t,s \in \partial X$ and $k = l(\mu) -
l(\nu),~~ m = l(\eta) - l(\beta)$.

If $l(\eta) \leq l(\nu)$ then $\nu = \eta \delta$ for some $\delta
\in X$.  Therefore $y = \eta \delta t$, implying $s = \delta t$,
hence $z = \beta \delta t$. Therefore $x \thicksim_r z$, where $r
= l(\mu) - l(\beta \delta) = l(\mu) - l(\beta) - l(\delta) =
l(\mu) - l(\beta) - (l(\nu) - l(\eta)) = [l(\mu) - l(\nu)] +
[l(\eta) - l(\beta)] = k + m$.

Similarly, if $l(\eta) \geq l(\nu)$ we get $x \thicksim_r z$,
where $r = k + m$.

\end{enumerate}

\begin{definition}

Let $\mathcal{G} := \{(x,k,y) \in  \partial X \times \Lambda \times
\partial X : x \thicksim_k y \}$.

For pairs in $\mathcal{G}^2 := \{((x,k,y),(y,m,z)) :
(x,k,y),(y,m,z) \in \mathcal{G}\}$, we define

\begin{equation}
\label{eq1:eps} (x,k,y) \cdot (y,m,z) = (x, k+m, z).
\end{equation}

For arbitrary $(x,k,y) \in \mathcal{G}$, we define
\begin{equation} \label{eq2:eps}(x,k,y)^{-1} = (y, -k, x). \end{equation}
\end{definition}

With the operations \eqref{eq1:eps} and \eqref{eq2:eps}, and source and range maps $s,r:\mathcal{G} \longrightarrow \partial X$ given
by $s(x,k,y) = y, ~~ r(x,k,y) = x,~~ \mathcal{\mathcal{G}}$ is a
groupoid with unit space $\partial X$.

We want to make $\mathcal{G}$ a locally compact $r$-discrete
groupoid with (topological) unit space $\partial X$.

For $p,q \in X$ and $A \in \mathcal{E}$, define $[p,q]_A = \{(px, l(p)-l(q), qx):x \in A\}.$

\begin{lemma} \label{lemm.501} For $p,q,r,s \in X$ and $A, B \in
\mathcal{E}$,
\begin{displaymath}
[p,q]_A \cap [r,s]_B = \left\{ \begin{array}{ll}
        [p,q]_{A\cap\mu B} & \textrm{if there exists $\mu \in X$ such that $r = p\mu,~~ s=q\mu$}\\
        {[r,s]_{(\mu A) \cap B}} & \textrm{if there exists $\mu \in X$ such that $p = r\mu,~~ q=s\mu$}\\
        \emptyset & \textrm{otherwise}.
        \end{array} \right.
\end{displaymath}

\end{lemma}

\begin{proof}
Let $t \in [p,q]_A \cap [r,s]_B$. Then $t = (px, k, qx) = (ry, m,
sy)$ for some $x \in A, y \in B$.  Clearly $k = m$. Furthermore,
$px = ry$ and $qx = sy$.  Suppose that $l(p) \leq l(r)$. Then $r =
p\mu$ for some $\mu \in X$, hence $px = p\mu y$, implying $x = \mu
y$. Hence $qx = q\mu y = sy$, implying $q\mu = s$.
Therefore $t = (px,k,qx) = (p\mu y, k ,q\mu y)$, that is, $t =
(px, k ,qx)$ for some $x \in A \cap \mu B$.
The case when $l(r) \leq l(p)$ follows by symmetry.
The reverse containment is clear.
\end{proof}

\begin{proposition}
Let $\mathcal{G}$ have the relative topology inherited from $\partial X \times \Lambda
\times \partial X$. Then $\mathcal{G}$ is a locally compact Housdorff groupoid, with base $\mathcal{D} = \{[a,b]_A:a,b \in X, A \in \mathcal{E}\}$ consisting of compact open subsets.
\end{proposition}

\begin{proof}
That $\mathcal{D}$ is a base follows from Lemma \ref{lemm.501}.
$[a,b]_A$ is a closed subset of $aA \times \{l(a)-l(b)\}\times
bA$, which is a compact open subset of $\partial X \times \Lambda
\times \partial X$.  Hence $[a,b]_A$ is compact open in $\mathcal{G}$.

To prove that inversion is continuous, let
$\phi:\mathcal{G} \longrightarrow \mathcal{G}$ be the inversion
function.  Then $\phi^{-1}([a,b]_A) = [b,a]_A$.  Therefore $\phi$
is continuous.  In fact $\phi$ is a homeomorphism.

For the product function, let $\psi:\mathcal{G}^2 \longrightarrow
\mathcal{G}$ be the product function.  Then
$\displaystyle{\psi^{-1}([a,b]_A) = \bigcup_{c \in X}(([a,c]_A
\times [c,b]_A) \cap \mathcal{G}^2)}$ which is open (is a union of
open sets).
\end{proof}

\begin{remark} \label{remk.501} We remark the following points:

\begin{enumerate}

\item [(a)] Since the set $\mathcal{D}$ is countable, the topology
is second countable.

\item [(b)] We can identify the unit space, $\partial X$, of
$\mathcal{G}$
 with the subset $\{(x,0,x):x \in
\partial X\}$ of $\mathcal{G}$ via $x \mapsto (x,0,x)$.  The
topology on $\partial X$ agrees with the topology it inherits by
viewing it as the subset $\{(x,0,x):x \in \partial X\}$ of
$\mathcal{G}$.
\end{enumerate}

\end{remark}

\begin{proposition} \label{prop.501} For each $A \in \mathcal{E}$
and each $a,b \in X$, $[a,b]_A$ is a $\mathcal{G}$-set.
$\mathcal{G}$ is $r$-discrete.
\end{proposition}
\begin{proof}

\begin{displaymath}
 \begin{array}{lll}
        & [a,b]_A & = \{(ax, l(a)-l(b), bx):x \in A\}\\
        \Rightarrow  & ([a,b]_A)^{-1} & = \{(bx, l(b)-l(a), ax):x \in
        A\}.
        \end{array}
\end{displaymath}

Hence, $((ax, l(a) - l(b), bx)(by, l(b) - l(a), ay)) \in [a,b]_A \times
([a,b]_A)^{-1}~~\bigcap~~ \mathcal{G}^2$ if and only if $x = y$.  And in that case, $(ax, l(a) - l(b), bx)\cdot(bx, l(b) - l(a), ax)
= (ax, 0, ax) \in \partial X$, via the identification stated in
Remark \ref{remk.501} (b).  This gives $[a,b]_A \cdot ([a,b]_A)^{-1} \subseteq \partial X$.
Similarly, $([a,b]_A)^{-1} \cdot [a,b]_A \subseteq \partial X$.
Therefore $\mathcal{G}$ has a base of compact open $\mathcal{G}$-sets,
implying $\mathcal{G}$ is $r$-discrete.
\end{proof}

Define $C^*(\Lambda)$ to be the $C^*$ algebra of the groupoid $\mathcal{G}$.  Thus $C^*(\Lambda) = \overline{span}\{\chi_S:S \in \mathcal{D}\}$.\\

\noindent For $A = V(p) \in \mathcal{E}$,
\begin{align*}
        [a,b]_A = [a,b]_{V(p)} & = \{(ax, l(a)-l(b), bx):x \in V(p)\}\\
         & = \{(ax, l(a)-l(b), bx):x=pt,~~ t \in
\partial X\}\\
         & = \{(apt, l(a)-l(b), bpt):t \in
\partial X\}\\
         & = [ap,bp]_{\partial X}.
\end{align*}
And for $A = V(p;q) = V(p) \setminus V(q) \in \mathcal{E}$,
\begin{align*}
        [a,b]_A & = \{(ax, l(a)-l(b), bx):x \in V(p) \setminus V(q)\}\\
         &  = \{(ax, l(a)-l(b), bx):x \in V(p)\} \setminus
\{(ax, l(a)-l(b), bx):x \in V(q)\}\\
         &  = [ap,bp]_{\partial X} \setminus
[aq,bq]_{\partial X}.\\
\end{align*}

Denoting $[a,b]_{\partial X}$ by $U(a,b)$ we get:
$\mathcal{D} = \{U(a,b):a,b \in X\}~~ \bigcup~~ \{U(a,b) \setminus U(c,d):a,b,c,d \in X, a \preceq c, b \preceq d\}$.
Moreover $\chi_{U(a,b) \setminus U(c,d)} = \chi_{U(a,b)} -
\chi_{U(c,d)}$, whenever $a \preceq c, b \preceq d$.
This give us: $$C^*(\Lambda) = \overline{span}\{\chi_{U(a,b)}:a,b \in X\}.$$\\

\section{Generators and relations} \label{sect.z24}

For $p \in X$, let $s_p = \chi_{U(p,0)}$, where 0 is the empty
word.  Then:
\begin{align*}
        s_p^*(x,k,y) & = \overline{\chi_{U(p,0)}((x,k,y)^{-1})}\\
         &  = \chi_{U(p,0)}(y,-k,x)\\
         &  = \chi_{U(0,p)}(x,k,y).\\
\end{align*}
Hence $s_p^* = \chi_{U(0,p)}$.

And for $p,q \in X$,
\begin{align*}
        s_ps_q(x,k,y) & = \displaystyle{\sum_{y \sim_mz}\chi_{U(p,0)}((x,k,y)(y,m,z))~\chi_{U(q,0)}((y,m,z)^{-1})}\\
         &  = \displaystyle{\sum_{y \sim_mz}\chi_{U(p,0)}(x,k+m,z)~\chi_{U(q,0)}(z,-m,y)}.\\
\end{align*}
Each term in this sum is zero except when $x = pz$, with $k+m = l(p)$, and $z = qy$, with $l(q) = -m$.  Hence, $k = l(p) - m = l(p) + l(q)$, and $x =
pz = pqy$.  Therefore $s_ps_q(x,k,y) = \chi_{U(pq,0)}(x,k,y)$; that is, $s_ps_q = \chi_{U(pq,0)} = s_{pq}$.

Moreover,
\begin{align*}
        s_ps_q^*(x,k,y) & = \displaystyle{\sum_{y \sim_mz}\chi_{U(p,0)}((x,k,y)(y,m,z))~\chi_{U(0,q)}((y,m,z)^{-1})}\\
         &  = \displaystyle{\sum_{y \sim_mz}\chi_{U(p,0)}(x,k+m,z)~\chi_{U(0,q)}(z,-m,y)}\\
         &  = \displaystyle{\sum_{y \sim_mz}\chi_{U(p,0)}(x,k+m,z)~\chi_{U(q,0)}(y,m,z)}.\\
\end{align*}
Each term in this sum is zero except when $x = pz,~~ k+m = l(p),~~
y = qz$, and $l(q) = m$.  That is, $k = l(p) - l(q)$, and $x =
pz,~~ y = qz$.  Therefore $s_ps_q^*(x,k,y) = \chi_{U(p,q)}(x,k,y)$; that is, $s_ps_q^* = \chi_{U(p,q)}$.

Notice also that
\begin{align*}
        s_p^*s_q(x,k,y) & = \displaystyle{\sum_{y \sim_mz}\chi_{U(0,p)}((x,k,y)(y,m,z))~\chi_{U(q,0)}((y,m,z)^{-1})}\\
         &  = \displaystyle{\sum_{y \sim_mz}\chi_{U(0,p)}(x,k+m,z)~\chi_{U(q,0)}(z,-m,y)}.\\
\end{align*}
is non--zero exactly when $z = px,~~ l(p) = -(k+m),~~ z = qy,$ and
$l(q) = -m$, which implies that $px = qy,~~ l(p) = -k-m = -k+l(q)$.  This implies that $s_p^*s_q$ is non--zero only if either $p \preceq q$ or $q \preceq p$.

If $p \preceq q$ then there exists $r \in X$ such that $q = pr$.
But $-k = l(p) -l(q) \Rightarrow k = l(q)-l(p) = l(r)$.
And $qy = pry \Rightarrow x = ry$.  Therefore $s_p^*s_q = s_r$.
And if $q \preceq p$ then there exists $r \in X$ such that $p = qr$.  Then
$(s_p^*s_q)^* = s_q^*s_p = s_r$.  Hence $s_p^*s_q = s_r^*$.
In short,
\begin{displaymath}
s_p^*s_q = \left\{ \begin{array}{ll}
        s_r & \textrm{if $q = pr$}\\
        s_r^* & \textrm{if $p = qr$}\\
        0 & \textrm{otherwise}.
        \end{array} \right.
\end{displaymath}
We have established that
\begin{equation}
\label{eq10:eps} C^*(\Lambda) =
\overline{span}\{s_ps_q^*:p,q \in X\}.
\end{equation}

Let $\mathcal{G}_0 := \{(x,0,y) \in \mathcal{G}:x,y \in \partial
X\}$.  Then $\mathcal{G}_0$, with the relative topology, has the
basic open sets $[a,b]_A$, where $A \in \mathcal{E}, a,b \in X$
and $l(a) = l(b)$.
Clearly $\mathcal{G}_0$ is a subgroupoid of $\mathcal{G}$. And
\begin{align*}
        C^*(\mathcal{G}_0) & = \overline{span}\{\chi_{U(p,q)}:p,q \in X, l(p) = l(q)\}\\
         &  \subseteq \overline{span}\{\chi_{[p,q]_A}:p,q \in X, l(p) = l(q),
         A \subseteq \partial X\text{ is compact open}\}\\
         & \subseteq C^*(\mathcal{G}_0).
\end{align*}
The second inclusion is due to the fact that $[p,q]_A$ is compact
open whenever $A \subseteq \partial X$ is, hence
$\chi_{[a,b]_A} \in C_c(\mathcal{G}_0)\subseteq
C^*(\mathcal{G}_0)$.

We wish to prove that the $C^*$-algebra $C^*(\mathcal{G}_0)$ is an AF algebra. But first notice that for any $\mu \in X$,
$\displaystyle{V(\mu) = V(\mu e_n') \cup V(\mu e_n'')}$,
where $e_n' =e_n = (0, \ldots, 0, 1)\in G_1$ and $e_n'' =e_n \in G_2$.  

Take a basic open set
$\displaystyle{A = V(\mu)\setminus \left(\bigcup_{k=1}^{m_1}V(\nu_k)\right)}$.  It is possible to rewrite $A$ as $\displaystyle V(p)\setminus \left(\bigcup_{k=1}^{m_2}V(r_k)\right)$ with $\mu \neq p$.  Here is a relatively simple example (pointed out to the author by Spielberg):
$\displaystyle{V(\mu)\setminus V(\mu e_n') = V(\mu e_n'')}$,
where $e_n' = e_n \in G_1$ and $e_n'' = e_n \in G_2$.


\begin{lemma} \label{lemm.505}
Suppose $\displaystyle{A = V(\mu) \setminus \left(\bigcup_{k=1}^s
V(\mu \nu_k)\right) \neq \emptyset}$.  Then we can write $A$ as
$\displaystyle{A = V(p) \setminus \left(\bigcup_{k=1}^{m_1} V(p
r_k)\right)}$ where $l(p)$ is the largest possible, that is, if\\
$\displaystyle{A = V(q) \setminus \left(\bigcup_{j=1}^{m_2} V(q
s_j)\right)}$ then $l(q) \leq l(p)$.
\end{lemma}

\begin{proof}
We take two cases:\\
Case I. For each $k = 1, \ldots, s$, there exists $i \in \{1, \ldots, n-1\}$ with $l_i(\nu_k) \geq 1$.
Choose $p = \mu,~~ r_k = \nu_k~~ \textrm{for each } k$ (i. e., leave $A$
the way it is).  Suppose now that $\displaystyle{A = V(q) \setminus \left(\bigcup_{j=1}^{m_2}
V(q s_j)\right)}$ with $l(p) \leq l(q)$.  We will prove that $l(p)
= l(q)$.
Assuming the contrary, suppose $l(p) < l(q)$.  Let $x \in A
\Rightarrow x = q y$ for some $y \in
\partial X$. Since $q y \in \displaystyle{V(p) \setminus
\left(\bigcup_{k=1}^s V(p r_k)\right)}$, $p \preceq q y$. But
$l(p) < l(q) \Rightarrow p \preceq q$.
Let $q = p r$, since $p \neq q$, $r \neq 0$.  Let $r =
a_1a_2\ldots a_d$.  Either $a_1 \in G_1\setminus \{0\}$ or $a_1
\in G_2 \setminus \{0\}$. Suppose, for definiteness, $a_1 \in
G_1\setminus \{0\}$. Take $t = (0,\ldots, 0, \infty) \in \partial G_2$.
Since $l(r_k) > l(t)~~ \textrm{for each } k = 1,\ldots, s$, we get $p r_k
\npreceq p t~~ \textrm{for each } k = 1,\ldots, s$, moreover $p t \in
V(p)$. Hence $pt \in A$. But $pr \npreceq pt \Rightarrow q \npreceq pt \Rightarrow pt \notin V(q) \Rightarrow pt \notin
\displaystyle{V(q) \setminus \left(\bigcup_{j=1}^{m_2} V(q
s_j)\right)}$ which is a contradiction to $\displaystyle{A = V(q)
\setminus \left(\bigcup_{j=1}^{m_2} V(q s_j)\right)}$.  Therefore
$l(p) = l(q)$. In fact, $p = q$.\\

\noindent Case II.  There exists $k \in \{1, \ldots, s\}$ with $l_i(\nu_k) = 0$, for each $i = 1, \ldots, n-1$.  After
rearranging, suppose that $l_i(\nu_k) = 0$ for each $k = 1, \ldots,
\alpha$ and each $i = 1, \ldots, n-1$; and that for each $k = \alpha + 1, \ldots, s$, $l_i(\nu_k) \geq 1$ for some $i \leq n-1$. We can also assume that $l(\nu_1)$ is the largest of $l(\nu_k)$'s for $k \leq \alpha$.  Then
\begin{align*}
        A & = V(\mu) \setminus \left(\bigcup_{k=1}^s
V(\mu \nu_k)\right)\\
         &  = \left[V(\mu) \setminus \left(\bigcup_{k=1}^\alpha
V(\mu \nu_k)\right)\right] \bigcap \left[V(\mu) \setminus \left(\bigcup_{k=\alpha + 1}^s
V(\mu \nu_k)\right)\right].
\end{align*}
Let $me_n = l(\nu_1)$ which is non zero.  We will prove that if we can rewrite $A$ as $\displaystyle{V(q) \setminus \left(\bigcup_{k=1}^{m_2} V(q s_k)\right)}$ with $l(\mu) \leq l(q)$ then $q = \mu r$ with $0 \leq l(r) \leq m(e_n)$.

Clearly if $\mu \npreceq q$, then $A \cap V(q) = \emptyset$.
So, if $A \cap \displaystyle{V(q) \setminus \left(\bigcup_{k=1}^{m_2} V(q s_k)\right)} \neq \emptyset$ then $\mu \preceq q$.  Now let $q = \mu r$, and let $\nu_1 = a_1a_2\ldots a_d$. Observe that since for each $j$, $a_j \in \Lambda^+$ and that $\l_k(\nu_1) = 0$ for each $k \leq n-1$, we have $l_k(a_j) = 0$ for all $k \leq n-1$.  Also, by assumption, $l(\nu_1) > 0$, therefore either $a_d \in G_1\setminus \{0\}$ or
$a_d \in G_2 \setminus \{0\}$. Suppose, for definiteness, that $a_d \in
G_1\setminus \{0\}$.  Let $a_d' = a_d - e_n$ and let $\nu' =
a_1a_2\ldots a_d'$ (or just $a_1a_2\ldots a_{d-1}$, if $a_d' =
0$). If $V(\mu \nu') \cap A = \emptyset$ then we can replace $\nu_1$ by $\nu'$ in the expression of $A$ and and (after rearranging the $\nu_i's$) choose a new $\nu_1$.  Since $A \neq \emptyset$ this process of replacement must stop with $V(\mu\nu') \cap A \neq \emptyset$. Letting $e_n' = e_n \in G_1$ and $e_n'' = e_n \in G_2$, then $V(\mu\nu') = V(\mu \nu' e_n') \cup V(\mu \nu'e_n'') = V(\mu \nu_1) \cup V(\mu \nu'e_n'')$. Since $V(\mu\nu_1) \cap A = \emptyset$, $A \cap V(\mu \nu'e_n'') \neq \emptyset$ hence $\displaystyle \nu' e_n'' \notin \{\nu_1, \ldots, \nu_\alpha\}.$  Take $t' = (0,\ldots, 0,\infty) \in \partial G_1$ and $t'' = (0,\ldots, 0,\infty) \in \partial G_2$.  Then $\displaystyle \mu \nu'e_n''t',~~\mu \nu'e_n''t'' \in V(\mu) \setminus \left(\bigcup_{k=1}^\alpha
V(\mu \nu_k)\right)$.  Moreover, for each $\displaystyle k = \alpha + 1, \ldots, s$, we have $\displaystyle l(\nu'e_n''t'),~~ l(\nu'e_n''t'') < l(\nu_k)$, implying $\displaystyle \mu \nu'e_n''t',~~\mu \nu'e_n''t'' \in V(\mu) \setminus \left(\bigcup_{k=\alpha+1}^s
V(\mu \nu_k)\right)$.  Hence $\mu \nu'e_n''t',~~\mu \nu'e_n''t'' \in \displaystyle{V(q) \setminus \left(\bigcup_{k=1}^{m_2} V(q s_k)\right)}$.  Therefore $q \preceq \mu \nu'e_n'' \Rightarrow \mu r \preceq \mu \nu'e_n'' \Rightarrow 0 \leq l(r) \leq l(\nu'e_n'') = l(\nu') + e_n = me_n$.  Therefore there is only a finite possible $r$'s we can choose form. [In fact, since $r \preceq \nu'e_n''$, there are at most $m$ of them to choose from.]
\end{proof}

To prove that $C^*(\mathcal{G}_0)$ is an AF algebra, we start with a finite subset $\mathcal{U}$ of the generating set $\{\chi_{U(p,q)}:p,q \in X, l(p) = l(q)\}$ and show that there is a finite dimensional
$C^*$-subalgebra of $C^*(\mathcal{G}_0)$ that contains the set $\mathcal{U}$.

\begin{theorem} \label{thm.501} $C^*(\mathcal{G}_0)$ is an AF
algebra.
\end{theorem}
\begin{proof}
Suppose that $\mathcal{U} = \{\chi_{U(p_1,q_1)},\chi_{U(p_2,q_2)}, \ldots \chi_{U(p_s,q_s)}\}$
is a (finite) subset of the generating set of $C^*(\mathcal{G}_0)$. Let $$\mathcal{S} := \{V(p_1), V(q_1), V(p_2), V(q_2), \ldots,
V(p_s), V(q_s)\}.$$   We ``disjointize'' the set $\mathcal{S}$ as follows.  For a subset $\mathbf{F}$ of $\mathcal{S}$, write  $$\displaystyle{A_\mathbf{F}:=\bigcap_{A\in
\mathbf{F}}A \setminus \bigcup_{A \notin \mathbf{F}}A}.$$  Define $$\mathcal{C}:=\displaystyle{\{A_\mathbf{F} : \mathbf{F}
\subseteq \mathcal{S}\}}.$$ Clearly, the set $\mathcal{C}$ is a finite collection of pairwise
disjoint sets. A routine computation reveals that for any $E \in
\mathcal{S}$, $\displaystyle{E = \bigcup\{C\in \mathcal{C}: C \subseteq E\}}$.
It follows from  \eqref{eq11:eps} that for any $\mathbf{F}
\subseteq \mathcal{S}$, $\displaystyle{\bigcap_{A\in \mathbf{F}}A
= V(p)}$, for some $p \in X$, if it is not empty.
Hence,
\begin{align*}
        A_\mathbf{F}  =  \displaystyle V(p) \setminus \bigcup_{i =
1}^k V(pr_i)\\
\end{align*}
for some $p \in X$ and some $r_i \in X$.  Let $p_\mathbf{F} \in X$ be such that $A_\mathbf{F} =  \displaystyle V(p) \setminus \bigcup_{i =
1}^k V(pr_i)$  and $l(p_\mathbf{F})$ is maximum (as in Lemma \ref{lemm.505}). Then
\begin{align*}
        A_\mathbf{F} &  = p_\mathbf{F}\left(\partial X \setminus
\left(\bigcup_{i=1}^k{V(r_i)}\right)\right)\\
         & = p_\mathbf{F}C_\mathbf{F},
\end{align*}
where $\displaystyle C_\mathbf{F} = \partial X \setminus
\left(\bigcup_{i=1}^k{V(r_i)}\right)$.  Now $\displaystyle V(p_\alpha) = p_{F_1}C_{F_1} \cup p_{F_2}C_{F_2} \cup
\ldots \cup p_{F_k}C_{F_k}$  where $\{F_1, F_2, \ldots, F_k\} = \{F \subseteq \mathcal{S}:V(p_\alpha) \in F\}$.  Notice that $p_{F_i}C_{F_i} \subseteq V(p_\alpha)$ for each $i$, hence $p_\alpha \preceq p_{F_i}$.  Hence $p_{F_i}C_{F_i} = p_\alpha t_iC_{F_i}$, for some $t_i \in X$.  Therefore $V(p_\alpha) = p_\alpha U_1 \cup p_\alpha U_2 \cup \ldots \cup p_\alpha U_k$ where $U_i = t_iC_{F_i}$.  Similarly $V(q_\alpha) =  q_\alpha V_1 \cup q_\alpha V_2 \cup \ldots \cup q_\alpha V_m$, where each $q_\alpha V_i \in \mathcal{C}$ is subset of $V(q_\alpha)$.
Consider the set $$\displaystyle \mathcal{B} := \{[p,q]_{C \cap D}:pC, qD \in \mathcal{C} \text{ and }p = p_\alpha, q = q_\alpha, 1 \leq \alpha \leq s\}.$$  Since $\mathcal{C}$ is a finite collection, this collection
is finite too.  We will prove that $\mathcal{B}$ is pairwise disjoint.

Suppose $[p,q]_{C \cap D}~~ \bigcap~~ [p',q']_{C' \cap D'}$ is
non-empty. Clearly $p(C \cap D)~~ \bigcap~~ p'(C' \cap D') \neq
\emptyset$, and $q(C \cap D)~~ \bigcap~~ q'(C' \cap D') \neq
\emptyset$. Therefore, among other things,  $pC \cap p'C' \neq \emptyset$ and $qD \cap
q'D' \neq \emptyset$, but by construction, $\{pC, qD, p'C',
q'D'\}$ is pairwise disjoint.  Hence $pC = p'C'$ and $qD = q'D'$.
Suppose, without loss of generality, that $l(p) \leq l(p')$.  Then
$p' = p r$ and $q' = q s$ for some $r,s \in X$, hence $[p',
q']_{C' \cap D'} = [pr, qs]_{C' \cap D'}$. Let $(p x, 0 , q x) \in
[p,q]_{C\cap D} \bigcap [pr, qs]_{C' \cap D'}$. Then $p x = p r t$
and $q x = q s t$, for some $t \in C' \cap D'$, hence $x = r t = s
t$. Therefore $r = s$ (since $l(r) = l(p') - l(p) = l(q') - l(q) =
l(s)$).  Hence $pC = p'C' = prC'$, and $qD = q'D' = qrD'$, implying $C
= rC'$ and $D = rD'$.  This gives us $C \cap D = rC' \cap rD' = r(C'\cap
D')$.  Hence $[p',q']_{C' \cap D'} = [pr,qr]_{C' \cap D'} =
[p,q]_{r(C' \cap D')} = [p,q]_{C \cap D}$. Therefore $\mathcal{B}$
is a pairwise disjoint collection.

For each $[p,q]_{C \cap D} \in \mathcal{B}$, since $C \cap D$
is of the form $\displaystyle{V(\mu) \setminus \bigcup_{i = 1}^k
V(\mu \nu_i)}$, we can rewrite $C \cap D$ as $\mu W$, where $W = \displaystyle{\partial X \setminus \bigcup_{i = 1}^k
V(\nu_i)}$ and $l(\mu)$ is maximal (by Lemma
\ref{lemm.505}).  Then $[p,q]_{C \cap D} = [p,q]_{\mu W} =
[p\mu,q\mu]_W$.  Hence each $[p,q]_{C \cap D} \in \mathcal{B}$ can be
written as $[p,q]_W$ where $l(p) = l(q)$ is maximal and $W = \displaystyle{\partial X \setminus \bigcup_{i = 1}^k
V(\nu_i)}$.

Consider the collection $\displaystyle{\mathcal{D} := \{ \chi_{[p,q]_W}:[p,q]_W \in \mathcal{B}}\}$. We
will show that, for each $1 \leq \alpha \leq s$,
$\chi_{U(p_\alpha,q_\alpha)}$ is a sum of elements of
$\mathcal{D}$ and that $\mathcal{D}$ is a self-adjoint system of
matrix units.
For the first, let $V(p_\alpha) = p_\alpha U_1 \cup p_\alpha U_2 \cup
\ldots \cup p_\alpha U_k$ and $V(q_\alpha) =
 q_\alpha V_1 \cup q_\alpha V_2 \cup \ldots
\cup q_\alpha V_m$.
One more routine computation gives us:
\begin{align*}
        U(p_\alpha,q_\alpha) = [p_\alpha,q_\alpha]_{\partial X} & = \displaystyle{\bigcup_{i, j=1}^{k,m}\big([p_\alpha,p_\alpha]_{U_i}
        \cdot [p_\alpha,q_\alpha]_{\partial X} \cdot [q_\alpha,q_\alpha]_{V_j}\big)}\\
         & = \displaystyle{\bigcup_{i, j=1}^{k,m}[p_\alpha,q_\alpha]_{U_i \cap
         V_j}}.
\end{align*}
Since the union is disjoint, $\displaystyle{\chi_{U(p_\alpha,q_\alpha)} =
\sum_{i,j=1}^{k,m}\chi_{[p_\alpha,q_\alpha]_{U_i \cap V_j}}}$.
And each $\chi_{[p_\alpha,q_\alpha]_{U_i \cap V_j}}$ is in the
collection $\mathcal{D}$.  Therefore
$\mathcal{U} \subseteq span(\mathcal{D})$.

To show that $\mathcal{D}$ is a self-adjoint system of matrix
units, let $\chi_{[p,q]_W}, \chi_{[r,s]_V} \in \mathcal{D}$.
Then
\begin{align*}
        \chi_{[p,q]_W} \cdot \chi_{[r,s]_V}(x_1,0,x_2) & = \displaystyle{\sum_{y_1, y_2}
        \chi_{[p,q]_W}\big((x_1, 0, x_2)(y_1, 0, y_2)\big) \cdot \chi_{[r,s]_V}(y_2, 0, y_1)}\\ \\
         & = \displaystyle{\sum_{y_2}
        \chi_{[p,q]_W}(x_1, 0, y_2) \cdot \chi_{[r,s]_V}(y_2, 0,
        x_2)},
\end{align*}
where the last sum is taken over all $y_2$ such that $x_1 \thicksim_0
y_2 \thicksim_0 x_2$.
Clearly the above sum is zero if $x_1 \notin pW$ or $x_2 \notin
sV$.  Also, recalling that $qW$ and $rV$ are either equal or
disjoint, we see that the above sum is zero if they are disjoint.
For the preselected $x_1$, if $x_1 = pz$ then $y_2 = qz$ (is
uniquely chosen). Therefore the above sum is just the single term
$\chi_{[p,q]_W}(x_1, 0, y_2) \cdot \chi_{[r,s]_V}(y_2, 0, x_2)$.
Suppose that $qW = rV$.  We will show that $l(q) = l(r)$, which
implies that $q = r$ and $W = V$.

Given this,
\begin{align*}
        \chi_{[p,q]_W} \cdot \chi_{[r,s]_V}(x_1,0,x_2) & = \chi_{[p,q]_W}(x_1, 0, y_2)
         \cdot \chi_{[r,s]_V}(y_2, 0, x_2)\\
         & = \chi_{[p,q]_W}(x_1, 0, y_2) \cdot \chi_{[q,s]_W}(y_2, 0, x_2)\\
         & = \chi_{[p,s]_W}(x_1, 0, x_2).
\end{align*}

To show that $l(q) = l(r)$, assuming the contrary, suppose $l(q) <
l(r)$ then $r = qc$ for some non-zero $c \in X$, implying $V =
cW$. Hence $[r,s]_V = [r,s]_{cW} = [rc,sc]_W$, which contradicts
the maximality of $l(r) = l(s)$.  By symmetry $l(r) < l(q)$ is
also impossible.  Hence $l(q) = l(r)$ and $W = V$. This concludes
the proof.
\end{proof}

\section{Crossed product by the gauge action}

Let $\hat{\Lambda}$ denote the dual of $\Lambda$, i.e., the
abelian group of continuous homomorphisms of $\Lambda$ into the
circle group $\mathds{T}$ with pointwise multiplication: for $t,s
\in \hat{\Lambda}$,
$\langle\lambda, ts\rangle = \langle\lambda,
t\rangle\langle\lambda,s\rangle~~ \textrm{for each } \lambda \in
\Lambda$, where $\langle\lambda, t\rangle$ denotes the value of $t
\in \hat{\Lambda}$ at $\lambda \in \Lambda$.

Define an action called the \textbf{gauge action}: $\alpha : \hat{\Lambda} \longrightarrow Aut(C^*(\mathcal{G}))$ as follows. For $t \in \hat{\Lambda}$, first define $\alpha_t:C_c(\mathcal{G})
\longrightarrow C_c(\mathcal{G})$ by $\alpha_t(f)(x,\lambda, y) =
\langle \lambda,t\rangle f(x, \lambda, y)$ then extend
$\alpha_t:C^*(\mathcal{G}) \longrightarrow C^*(\mathcal{G})$ continuously.
Notice that $(A, \hat{\Lambda}, \alpha)$ is a $C^*$- dynamical system.

Consider the linear map $\Phi$ of $C^*(\mathcal{G})$ onto the fixed-point algebra $C^*(\mathcal{G})^\alpha$ given by $$\Phi(a) = \int_{\hat{\Lambda}}{\alpha_t(a) \,dt}\text{, for }a \in C^*(\mathcal{G}).$$  where $dt$ denotes a normalized Haar measure on $\widehat{\Lambda}$.

\begin{lemma} Let $\Phi$ be defined as above.
    \begin{enumerate}
        \item The map $\Phi$ is a faithful conditional expectation; in the sense that $\Phi(a^*a) = 0$ implies $a = 0$.
        \item $C^*(\mathcal{G}_0) = C^*(\mathcal{G})^\alpha$.
    \end{enumerate}
\end{lemma}
\begin{proof}
  Since the action $\alpha$ is continuous, we see that $\Phi$ is a conditional expectation from $C^*(\mathcal{G})$ onto $C^*(\mathcal{G})^\alpha$, and that the expectation is faithful.  For $p,q \in X$, $\alpha_t(s_ps_q^*)(x,l(p)-l(q),y) = \langle  l(p)-l(q),t\rangle s_ps_q^*(x, l(p)-l(q),y)$.  Hence if $l(p) = l(q)$ then  $\alpha_t(s_ps_q^*) = s_ps_q^* \textrm{ for each } t \in  \hat{\Lambda}$.  Therefore $\alpha$ fixes $C^*(\mathcal{G}_0)$.  Hence $C^*(\mathcal{G}_0) \subseteq C^*(\mathcal{G})^\alpha$.  By continuity of $\Phi$ it suffices to show that $\Phi(s_ps_q^*) \in C^*(\mathcal{G}_0)$ for all $p, q \in X$.
  \begin{align*}
        \int_{\hat{\Lambda}}{\alpha_t(s_ps_q^*) \,dt} &  = \int_{\hat{\Lambda}}{\langle
            l(p)-l(q), t\rangle s_ps_q^* \,dt} = 0, \textrm{ when }l(p) \neq l(q).
  \end{align*}
It follows from (\ref{eq10:eps}) that $C^*(\mathcal{G})^\alpha \subseteq C^*(\mathcal{G}_0)$.  Therefore $C^*(\mathcal{G})^\alpha = C^*(\mathcal{G}_0)$.
\end{proof}

 We study the crossed product $C^*(\mathcal{G}) \times_\alpha \widehat{\Lambda}$.  Recall that $C_c(\hat{\Lambda}, A)$, which is equal to
$C(\hat{\Lambda}, A)$, since $\hat{\Lambda}$ is compact, is a dense
*-subalgebra of $A \times_\alpha \hat{\Lambda}$. Recall also that
multiplication (convolution) and involution on $C(\hat{\Lambda},
A)$ are, respectively, defined by:
$$\displaystyle{(f\cdot g)(s) = \int_{\hat{\Lambda}}
{f(t)\alpha_t(g(t^{-1}s))\,dt}}$$
and
$$f^*(s)=\alpha(f(s^{-1})^*).$$

The functions of the form $f(t) =
\langle\lambda,t\rangle s_ps_q^*$ from $\hat{\Lambda}$ into $A$
form a generating set for $A \times_\alpha \hat{\Lambda}$. Moreover
the fixed-point algebra $C^*(\mathcal{G}_0)$ can be
imbedded into $A \times_\alpha \hat{\Lambda}$ as
follows:  for each $b \in C^*(\mathcal{G}_0)$, define the function $b:\hat{\Lambda}
\longrightarrow A$ as $b(t) = b$ (the constant function). Thus $C^*(\mathcal{G}_0)$ is a subalgebra of $A \times_\alpha \hat{\Lambda}$.

\begin{proposition}\label{prop.510}
The $C^*$-algebra $B := C^*(\mathcal{G}_0)$ is a hereditary $C^*$-subalgebra of $A \times_\alpha
\hat{\Lambda}$.
\end{proposition}
\begin{proof}
To prove the theorem, we prove that $B~~ \cdot~~ A \times_\alpha \hat{\Lambda}~~
\cdot~~ B\subseteq B$. Since $A \times_\alpha \hat{\Lambda}$ is
generated by functions of the form $f(t) = \langle\lambda,t\rangle
s_ps_q^*$, it suffices to show that $b_1\cdot f \cdot b_2 \in B$
whenever $b_1,b_2\in B$ and $f(t) = \langle\lambda,t\rangle
s_ps_q^*$ for $\lambda \in \Lambda, p,q \in X$.
\begin{align*}
    (b_1\cdot f \cdot b_2)(z) &  = \int_{\hat{\Lambda}}{b_1(t)
   \alpha_t((f\cdot b_2)(t^{-1}z))\,dt}\\
     & = \int_{\hat{\Lambda}}{b_1\alpha_t\left(\int_{\hat{\Lambda}}
     {f(w)\alpha_w(b_2(w^{-1}t^{-1}z)) \,dw}\right) \,dt}\\
     & =  \int_{\hat{\Lambda}}{\int_{\hat{\Lambda}}{b_1\alpha_t(f(w)
     \alpha_w(b_2))} \,dw \,dt}\\
     & = \int_{\hat{\Lambda}}{\int_{\hat{\Lambda}}{b_1
     \alpha_t(\langle\lambda,w\rangle s_ps_q^*)b_2} \,dw \,dt},
      \text{ since $\alpha_w(b_2)=\alpha_t(b_2) = b_2$}\\
     & = \int_{\hat{\Lambda}}{\int_{\hat{\Lambda}}{b_1\langle
     \lambda,w\rangle \alpha_t(s_ps_q^*)b_2} \,dw \,dt}\\
     & =  \int_{\hat{\Lambda}}{\int_{\hat{\Lambda}}
     {b_1\langle\lambda,w\rangle \langle l(p)-l(q),t\rangle
     s_ps_q^*b_2} \,dw \,dt}\\
     & = \int_{\hat{\Lambda}}{\langle\lambda,w\rangle
     \,dw}\int_{\hat{\Lambda}}{\langle l(p)-l(q),t\rangle \,dt}~~b_1s_ps_q^*b_2 \\
     & = 0 \text{ unless $\lambda = 0$ and $l(p)-l(q) =
      0$.}
\end{align*}
And in that case (in the case when $\lambda = 0$ and $l(p)-l(q) =
0$) we get $(b_1\cdot f \cdot b_2)(z) = b_1s_ps_q^*b_2 \in B$
(since $l(p) = l(q)$). Therefore $B$ is hereditary.
\end{proof}

Let $I_B$ denote the ideal in $A \times_\alpha \hat{\Lambda}$
generated by $B$.  The following corollary follows from Theorem
\ref{thm.501} and Proposition \ref{prop.510}.

\begin{corollary} \label{corr.510}
$I_B$ is an AF algebra.
\end{corollary}

We want to prove that $A \times_\alpha \hat{\Lambda}$ is an AF
algebra, and to do this we consider the dual system.  Define $\hat{\alpha}:\hat{\hat{\Lambda}} = \Lambda \longrightarrow
Aut(A \times_\alpha \hat{\Lambda})$ as follows:
For $\lambda \in \Lambda$ and $f \in C(\hat{\Lambda}, A)$, we
define $\hat{\alpha}_\lambda(f) \in C(\hat{\Lambda}, A)$ by:
$\hat{\alpha}_\lambda(f)(t) = \langle\lambda,t\rangle f(t)$.
Extend $\hat{\alpha}_\lambda$ continuously.

As before we use $\cdot$ to represent multiplication in $A
\times_\alpha \hat{\Lambda}$.

\begin{lemma} \label{lemm.450}
$\hat{\alpha}_\lambda(I_B) \subseteq I_B$ for each $\lambda \geq
0$.
\end{lemma}

\begin{proof}
Since the functions of the form $f(t) = \langle\lambda,t\rangle
s_ps_q^*$ make a generating set for $A \times_\alpha
\hat{\Lambda}$, it suffices to show that if $\lambda > 0$ then
$\hat{\alpha}_\lambda(f\cdot b \cdot g) \in I_B$ for $f(t) =
\langle\lambda_1,t\rangle s_{p_1}s_{q_1}^*$, $g(t) =
\langle\lambda_2,t\rangle s_{p_2}s_{q_2}^*$, and $b =
s_{p_0}s_{q_0}^*$, with $l(p_0) = l(q_0)$.\\
First
\begin{align*}
   (f\cdot b \cdot g)(z) &  = \int_{\hat{\Lambda}}{f(t) \alpha_t((b\cdot g)(t^{-1}z))\,dt}\\
     & = \int_{\hat{\Lambda}}{f(t) \alpha_t\left[\int_{\hat{\Lambda}}{b(w) \alpha_w(g(w^{-1}t^{-1}z))\,dw}\right]\,dt}\\
     & = \int_{\hat{\Lambda}}{f(t) \left[\int_{\hat{\Lambda}}{b \alpha_{tw}(g(w^{-1}t^{-1}z)\,dw)}\right]\,dt}\\
     & = \int_{\hat{\Lambda}}{f(t) \left[\int_{\hat{\Lambda}}{b \alpha_{w}(g(w^{-1}z)\,dw)}\right]\,dt}\\
     & = \int_{\hat{\Lambda}}{\int_{\hat{\Lambda}}{f(t)b \alpha_{w}(g(w^{-1}z))\,dw}\,dt}\\
     & = \int_{\hat{\Lambda}}{\int_{\hat{\Lambda}}{\langle\lambda_1,t\rangle s_{p_1}s_{q_1}^*s_{p_0}s_{q_0}^* \langle\lambda_2,w^{-1}z\rangle\alpha_{w}(s_{p_2}s_{q_2}^*)\,dw}\,dt}\\
     & = \int_{\hat{\Lambda}}{\int_{\hat{\Lambda}}{\langle\lambda_1,t\rangle s_{p_1}s_{q_1}^*s_{p_0}s_{q_0}^* \langle\lambda_2,w^{-1}z\rangle\langle l(p_2)-l(q_2),w\rangle s_{p_2}s_{q_2}^*\,dw}\,dt}.
\end{align*}
Hence
\begin{align*} & \hat{\alpha}_\lambda(f\cdot b \cdot g)(z)\\
& = \langle\lambda,z\rangle\int_{\hat{\Lambda}}{\int_{\hat{\Lambda}}{\langle\lambda_1,t\rangle s_{p_1}s_{q_1}^*s_{p_0}s_{q_0}^* \langle\lambda_2,w^{-1}z\rangle\langle l(p_2)-l(q_2),w\rangle s_{p_2}s_{q_2}^*\,dw}\,dt}\\
      & = \int_{\hat{\Lambda}}{\int_{\hat{\Lambda}}{\langle\lambda_1,t\rangle s_{p_1}s_{q_1}^*s_{p_0}s_{q_0}^*
     \langle\lambda,w^{-1}z\rangle\langle\lambda,w\rangle\langle\lambda_2,w^{-1}z\rangle\nonumber
     \langle l(p_2)-l(q_2),w\rangle s_{p_2}s_{q_2}^*\,dw}\,dt}\\
      & = \int_{\hat{\Lambda}}{\int_{\hat{\Lambda}}{\langle\lambda_1,t\rangle s_{p_1}s_{q_1}^*s_{p_0}s_{q_0}^* \langle\lambda+\lambda_2,w^{-1}z\rangle\langle \lambda+l(p_2)-l(q_2),w\rangle s_{p_2}s_{q_2}^*\,dw}\,dt}\\
\end{align*}

Letting $\lambda' = \lambda \in G_1$, then the last integral gives us:

\begin{align*}
    & = \int_{\hat{\Lambda}}\int_{\hat{\Lambda}}\langle\lambda_1,t\rangle s_{p_1}s_{q_1}^*s_{\lambda'}^* s_{\lambda'} s_{p_0}s_{q_0}^* s_{\lambda'}^* s_{\lambda'}\langle\lambda+\lambda_2,w^{-1}z\rangle\\
    & \hspace{2.5 in}\langle \lambda+l(p_2)-l(q_2),w\rangle s_{p_2}s_{q_2}^*\,dw\,dt\\
    & = \int_{\hat{\Lambda}}\int_{\hat{\Lambda}}\langle\lambda_1,t\rangle s_{p_1}s_{\lambda' q_1}^* s_{\lambda' p_0}s_{\lambda' q_0}^* \langle\lambda+\lambda_2,w^{-1}z\rangle\\
    & \hspace{2.5in} \langle \lambda+l(p_2)-l(q_2),w\rangle s_{\lambda' p_2}s_{q_2}^*\,dw\,dt\\
    & = (f' \cdot b' \cdot g')(z),\\
\end{align*}
where $f'(t) = \langle\lambda_1,t\rangle s_{p_1}s_{\lambda'
q_1}^*,~~ g'(t) = \langle\lambda + \lambda_2,t\rangle s_{\lambda'
p_2}s_{q_2}^*$, and $b' = s_{\lambda' p_0}s_{\lambda' q_0}^*$.
Therefore $\displaystyle \hat{\alpha}_\lambda(f\cdot b \cdot g) \in I_B$.
\end{proof}

For each $\lambda \in \Lambda$ define $I_\lambda :=
\hat{\alpha}_\lambda(I_B)$.  Clearly each $I_\lambda$ is an ideal
of $A \times_\alpha \hat{\Lambda}$ and is an AF algebra.
Let $\lambda_1 < \lambda_2$ then $\lambda_2 - \lambda_1 > 0
\Rightarrow I_{\lambda_2 - \lambda_1} =
\hat{\alpha}_{\lambda_2-\lambda_1}(I_B) \subseteq I_B$.  Therefore
$I_{\lambda_2} = \hat{\alpha}_{\lambda_2}(I_B) =
\hat{\alpha}_{\lambda_1}(\hat{\alpha}_{\lambda_2-\lambda_1}(I_B))
\subseteq I_{\lambda_1}$.  That is, $I_{\lambda_1} \supseteq
I_{\lambda_2}$ whenever $\lambda_1 < \lambda_2$.  In particular
$I_B = I_0 \supseteq I_\lambda~~ \textrm{for each } \lambda \geq
0.$  Furthermore, if $f \in C(\hat{\Lambda}, A)$ is given by $f(t) =
\langle\lambda,t\rangle s_ps_q^*$, for $\lambda \in \Lambda$ and
$p,q \in X$ then $\hat{\alpha}_\beta(f)(t) = \langle\beta,t\rangle
f(t) = \langle\beta,t\rangle\langle\lambda,t\rangle s_ps_q^* =
\langle\beta+\lambda,t\rangle s_ps_q^*$.

For $f \in C(\mathcal{G})$ given by $f(t) = \langle\lambda,
t\rangle s_ps_q^*$,  let us compute $f^*$, $f \cdot f^*$, and $f^*
\cdot f$ so we can use them in the next lemma.
\begin{align*}
   f^*(t) &  = \alpha_t(f(t^{-1})^*)\\
           & = \alpha_t\left(\left(\langle\lambda, t^{-1}\rangle s_ps_q^*\right)^*\right)\\
           & = \overline{\langle\lambda, t^{-1}\rangle} \alpha_t\left(s_qs_p^*\right)\\
           & = \langle\lambda, t\rangle \langle l(q)-l(p), t\rangle s_qs_p^*\\
           & = \langle\lambda + l(q)-l(p), t\rangle s_qs_p^*,\\
\end{align*}

\begin{align*}
   (f \cdot f^*)(z) &  = \int_{\hat{\Lambda}}{f(t) \alpha_t(f^*(t^{-1}z))\,dt}\\
     & = \int_{\hat{\Lambda}}{\langle \lambda, t\rangle s_ps_q^* \alpha_t\left(\langle\lambda+l(q)-l(p),t^{-1}z\rangle s_qs_p^*\right)\,dt}\\
     & = \int_{\hat{\Lambda}}{\langle \lambda, t\rangle s_ps_q^* \langle\lambda + l(q)-l(p),t^{-1}z\rangle \langle l(q)-l(p),t\rangle s_qs_p^*\,dt}\\
     & = \int_{\hat{\Lambda}}{\langle \lambda+l(q)-l(p), t\rangle s_ps_q^* \langle\lambda + l(q)-l(p),t^{-1}z\rangle  s_qs_p^*\,dt}\\
     & = \int_{\hat{\Lambda}}{\langle \lambda+l(q)-l(p), z\rangle s_ps_q^* s_qs_p^*\,dt}\\
     & = \langle \lambda+l(q)-l(p), z\rangle s_ps_q^*s_qs_p^*\\
     & = \langle \lambda+l(q)-l(p), z\rangle s_ps_p^*,\\
\end{align*}
and
\begin{align*}
   (f^* \cdot f)(z) &  = \langle \big(\lambda + l(q)-l(p)\big) + l(p)-l(q), z\rangle s_qs_q^*\\
     & = \langle \lambda, z\rangle s_qs_q^*.\\
\end{align*}

\begin{lemma} \label{lemm.520}
Let $\lambda \in \Lambda$, $p,q \in X$, $f(t) =
\langle\lambda,t\rangle s_qs_q^*$, and let $g(t) =
\langle\lambda,t\rangle s_ps_q^*$.
\begin{enumerate}
\item[(a)] If $\lambda \geq 0$ then $f \in I_B$.

\item[(b)] If $\lambda + l(q) \geq l(p)$ then $g \in I_B$.

\end{enumerate}
\end{lemma}

\begin{proof}
To prove (a),  $s_qs^*_q \in C^*(\mathcal{G}_0) \subseteq I_B$. Then $f \in I_B$, since $\lambda \geq 0$, by Lemma \ref{lemm.450}.
To prove (b), $(g \cdot g^*)(z) = \langle \lambda + l(q) -
l(p),z\rangle s_ps_p^*$.  By (a), $g \cdot g^* \in I_B$, implying
$g \in I_B$.
\end{proof}

\begin{theorem} $A \times_\alpha \hat{\Lambda}$ is an AF algebra.
\end{theorem}
\begin{proof}
Let $f(t) = \langle \lambda, t\rangle s_ps_q^*$.  Choose $\beta
\in \Lambda$ large enough such that $\beta + \lambda + l(q) \geq
l(p)$.  Then
\begin{align*}
   \hat{\alpha}_\beta(f)(z) &  = \langle \beta, z\rangle \langle \lambda, z\rangle s_ps_q^*\\
     & = \langle \beta + \lambda, z\rangle s_ps_q^*.\\
\end{align*}

Applying Lemma \ref{lemm.520} (b),  $\hat{\alpha}_\beta(f) \in
I_B$.  Thus
$\hat{\alpha}_{-\beta}\left(\hat{\alpha}_\beta(f)\right) \in
I_{-\beta}$, implying $f \in I_{-\beta}$.
Therefore $\displaystyle{A \times_\alpha \hat{\Lambda} =
\overline{\bigcup_{\lambda \leq 0}{I_\lambda}}}$.  Since each
$I_\lambda$ is an AF algebra, so is $A \times_\alpha
\hat{\Lambda}$.
\end{proof}

\section{Final results}

Let us recall that an $r$-discrete groupoid $G$ is \textit{locally
contractive} if for every nonempty open subset $U$ of the unit
space there is an open $G$-set $Z$ with $s(\bar{Z}) \subseteq U$
and $r(\bar{Z}) \subsetneqq s(Z)$.  A subset $E$ of the unit space
of a groupoid $G$ is said to be invariant if its saturation $[E] =
r(s^{-1}(E))$ is equal to $E$.

An $r$-discrete groupoid $G$ is
\textit{essentially free} if the set of all $x$'s in the unit
space $G^0$ with $r^{-1}(x) \cap s^{-1}(x) = \{x\}$ is dense in
the unit space.  When the only open invariant subsets of $G^0$ are
the empty set and $G^0$ itself, then we say that $G$ is
\textit{minimal}.

\begin{lemma} \label{lemm.555}
$\mathcal{G}$ is locally contractive, essentially free and
minimal.
\end{lemma}
\begin{proof}
To prove that $\mathcal{G}$ is locally contractive, let $U
\subseteq \mathcal{G}^0$ be nonempty open.  Let $V(p;q) \subseteq
U$.  Choose $\mu \in X$ such that $p \preceq \mu$, $q \npreceq
\mu$ and $\mu \npreceq q$.  Then $V(\mu) \subseteq V(p;q)
\subseteq U$. Let $Z = [\mu,0]_{V(\mu)}$.  Then $Z = \bar Z,~~
s(Z) = V(\mu) \subseteq U,~~ r(Z) = \mu V(\mu) \subsetneqq V(\mu)
\subseteq U$. Therefore $\mathcal{G}$ is locally contractive.

To prove that $\mathcal{G}$ is essentially free, let $x \in
\partial X$.  Then $r^{-1}(x) = \{(x,k,y):x \sim_k y \}$ and
$s^{-1}(x) = \{(z,m,x):z \sim_m x \}$.  Hence $r^{-1}(x) \cap
s^{-1}(x) = \{(x,k,x):x \sim_k x\}$.  Notice that $r^{-1}(x) \cap
s^{-1}(x) = \{x\}$ exactly when $x \sim_k x$ which implies $k = 0$.
If $k \neq 0$ then $x = pt = qt$, for some $p,q \in X$, $t \in
\partial X$ such that $l(p) - l(q) = k$.
If $k > 0$ then $l(p) > l(q)$ and we get $q \preceq p$.  Hence $p
= qb$, for some $b \in X \setminus \{0\}$.  Therefore $x = qbt =
qt$, implying $bt = t$.  Hence $x = qbbb\ldots$.  Similarly,
if $k < 0$ then $x = pbbb\ldots$, for some $b \in X$, with $l(b) >
0$.
Therefore, to prove that $\mathcal{G}$ is essentially free, we
need to prove that if $U$ is an open set containing an element of
the form $pbbb\ldots$, with $l(b) > 0$, then it contains an
element that cannot be written in the form of $qddd\ldots$, with
$l(d) > 0$.
Suppose $pbbb\ldots \in U$, where $U$ is open in $\mathcal{G}^0$.
then $U \supseteq V(\mu;\nu)$ for some $\mu,\nu \in X$.  Choose
$\eta \in X$ such that $\mu \preceq \eta,~~ \nu \npreceq \eta$,
and $\eta \npreceq \nu$.  Then $V(\eta) \subseteq V(\mu;\nu)
\subseteq U$. Now take  $a_1 = (1,0,\ldots, 0) \in G_1,~~ a_2 = (2,0,\ldots, 0) \in G_2,~~
a_3 = (3,0,\ldots, 0) \in G_1,~~ a_4 = (4,0,\ldots, 0) \in G_2$, etc.
Now $t = \eta a_1a_2a_3 \ldots \in V(\eta)
\subseteq U$, but $t$ cannot be written in the form of
$qddd\ldots$.

To prove that $\mathcal{G}$ is minimal, let $E \subseteq
\mathcal{G}^0$ be nonempty open and invariant, i.e., $E =
r(s^{-1}(E))$.  We want to show that $E = \mathcal{G}^0$.  Since
$E$ is open, there exist $p,q \in X$ such that $V(p;q) \subseteq
E$. But
$$s^{-1}(V(p;q)) = \{(\mu x, l(\mu) - l(p\nu), p\nu x): q \npreceq
p\mu x\}.$$
Let $x \in \mathcal{G}^0$.  Choose $\nu \in X$ such that $p\nu
\npreceq q$ and $q \npreceq p\nu$.  Then $(x, -l(p\nu), p\nu x)
\in s^{-1}(V(p;q)) \subseteq s^{-1}(E)$ and $r(x, -l(p\nu), p\nu x) = x$.  That is, $x \in r(s^{-1}(V(p;q)))$, hence $E =
\mathcal{G}^0$.  Therefore $\mathcal{G}$ is minimal.
\end{proof}

\begin{proposition} \label{prop.AD24} \cite[Proposition
2.4]{AD}

Let $G$ be an $r$-discrete groupoid, essentially free and locally contractive. Then every non-zero hereditary $C^*$-subalgebra of $C^*_r(G)$ contains an infinite projection.
\end{proposition}

\begin{corollary} \label{corr.551}
$C^*_r(\mathcal{G})$ is simple and purely infinite.
\end{corollary}
\begin{proof}
This follows from Lemma \ref{lemm.555} and Proposition \ref{prop.AD24}.
\end{proof}

\begin{theorem}
$C^*(\mathcal{G})$ is simple, purely infinite, nuclear and
classifiable.
\end{theorem}
\begin{proof}
It follows from Takesaki-Takai Duality Theorem that $C^*(\mathcal{G})$ is stably
isomorphic to $C^*(\mathcal{G}) \times_\alpha \hat{\Lambda} \times_{\hat{\alpha}}
\Lambda$.  Since $C^*(\mathcal{G}) \times_\alpha \hat{\Lambda}$ is an AF algebra and that
$\Lambda = \mathds{Z}^2$ is amenable, $C^*(\mathcal{G})$ is nuclear and
classifiable.  We prove that $C^*(\mathcal{G}) = C_r^*(\mathcal{G})$. From Theorem \ref{thm.501} we get that the fixed-point algebra $C^*(\mathcal{G}_0)$ is an AF algebra.  The inclusion $C_c(\mathcal{G}_0) \subseteq C_c(\mathcal{G}) \subseteq C^*(\mathcal{G})$ extends to an injective $*$-homomorphism $C^*(\mathcal{G}_0) \subseteq C^*(\mathcal{G})$ (injectivity follows since $C^*(\mathcal{G}_0)$ is an AF algebra).  Since $C^*(\mathcal{G}_0) = C_r^*(\mathcal{G}_0)$, it follows that $C^*(\mathcal{G}_0) \subseteq C^*_r(\mathcal{G})$.  Let $E$ be the conditional expectation of $C^*(\mathcal{G})$ onto $C^*(\mathcal{G}_0)$ and $\lambda$ be the canonical map of $C^*(\mathcal{G})$ onto $C^*_r(\mathcal{G})$.  IF $E^r$ is the conditional expectation of $C^*_r(\mathcal{G})$ onto $C^*(\mathcal{G}_0)$, then $E^r \circ \lambda = E$.  It then follows that $\lambda$ is injective.  Therefore $C^*(\mathcal{G}) = C_r^*(\mathcal{G})$.  Simplicity and pure infiniteness follow from Corollary \ref{corr.551}.
\end{proof}

\begin{remark}
Kirchberg-Phillips classification theorem states that
simple, unital, purely infinite, and nuclear $C^*$-algebras are
classified by their $K$-theory \cite{P}.  In the continuation of this project, we wish to compute the $K$-theory of $C^*(\mathds{Z}^n)$.

Another interest is to generalize the study and/or the result to a more general ordered group or even a ``larger" group, such as $\mathds{R}^n$
\end{remark}


\end{document}